\documentclass[11pt,dvips]{article}
\usepackage{theorem}
\usepackage{groups}
\usepackage{epsfig}
\usepackage{amsfonts}
\usepackage{amssymb}
\usepackage{latexsym}
\usepackage[centertags]{amsmath}
\begin{document}
\maketitle{ENGEL ELEMENTS IN GROUPS}{%
ALIREZA ABDOLLAHI}{Department of Mathematics, University of Isfahan, Isfahan 81746-73441, IRAN.\newline
  Email: a.abdollahi@math.ui.ac.ir}
\begin{abstract}
We give a survey of results on the structure of right and left Engel elements of a group. We also present some  new results in this topic.
\end{abstract}
\section{Introduction}
Let $G$ be any group and $x,y\in G$. Define inductively the $n$-Engel left normed commutator $$[x,_n y]=[x,\underset{n}{\underbrace{y,\dots,y}}]$$ of the pair $(x,y)$ for a given non-negative integer $n$, as follows:
$$[x,_0 y]:=x, \; [x,_1 y]:=[x,y]=x^{-1}y^{-1}xy=:x^{-1}x^y,$$ and  for  all $n>0$ $$[x,_n y]=[[x,_{n-1} y],y].$$
An element $a\in G$ is called left  Engel whenever for every element $g\in G$ there exists a non-negative integer $n=n(g,a)$ possibly
depending on the elements $g$ and $a$ such that $[g, _n a]=1$.  For a positive integer $k$,  an element $a\in G$ is
called a left  $k$-Engel element of $G$ whenever $[g,_k a]=1$ for all $g\in G$. An element $a\in G$ is called a bounded left
Engel element if it is left $k$-Engel for some $k$.
We denote by $L(G)$,
$L_k(G)$  and $\overline{L}(G)$, the set of left Engel elements, left  $k$-Engel elements, bounded left Engel elements of $G$, respectively.

In  definitions of various types of  left Engel elements $a$ of a group $G$, we observe that
 the variable element $g$ appears on the left hand side of the element $a$. So similarly one can define  various types of  right Engel
 elements $a$ in a group by letting  the variable element $g$ to appear (in the $n$-Engel commutator of $a$ and $g$) on the right hand side of the element $a$.
Therefore, an element $a\in G$ is called right  Engel whenever for every element $g\in G$ there exists a non-negative integer $n=n(g,a)$  such that $[a, _n g]=1$.  For a positive integer $k$,  an element $a\in G$ is
called a right  $k$-Engel element of $G$ whenever $[a,_k g]=1$ for all $g\in G$. An element $a\in G$ is called a bounded right
Engel element if it is right $k$-Engel for some $k$.
We will denote   $R(G)$,
$R_k(G)$  and $\overline{R}(G)$, the set of right Engel elements, right  $k$-Engel elements, bounded right Engel elements of $G$, respectively.\\
All these subsets are invariant under  automorphisms of $G$.\\
Groups in which all elements are left Engel   are called Engel groups and for a given positive integer $n$, a group is called $n$-Engel if all of whose elements are left $n$-Engel elements.  It is clear that  $$G=L(G) \Leftrightarrow G=R(G) \;\;\text{and}\;\; G=L_n(G) \Leftrightarrow G=R_n(G).$$
 A group is called bounded Engel if it is $k$-Engel for some positive integer $k$. Note that
$$L_1(G)\subseteq L_2(G)\subseteq \cdots \subseteq L_n(G)\subseteq \cdots\subseteq L(G) \;\;\;\text{and}\;\;\; \overline{L}(G)=\bigcup_{k\in\mathbb{N}}L_k(G)$$
$$R_1(G)\subseteq R_2(G)\subseteq \cdots \subseteq R_n(G)\subseteq \cdots \subseteq R(G)\;\;\;\text{and}\;\;\; \overline{R}(G)=\bigcup_{k\in\mathbb{N}}R_k(G)$$

As stated in \cite[p. 41 of Part II]{Robinson72-2}, the major goal of Engel theory can be stated as follows: to find conditions on a group $G$ which will ensure that $L(G)$, $\overline{L}(G)$, $R(G)$ and $\overline{R}(G)$ are subgroups and, if possible, coincide with the Hirsch-Plotkin radical, the Baer radical, the hypercenter and the $\omega$-center respectively. So let us put forward the following question.
\begin{qu}\label{q1}
For which groups $G$ and which positive integers $n$
\begin{enumerate}
\item $L(G)$ is a subgroup of $G$?
\item $R(G)$ is a subgroup of $G$?
\item $\overline{L}(G)$ is a subgroup of $G$?
\item $\overline{R}(G)$ is a subgroup of $G$?
\item $L_n(G)$ is a subgroup of $G$?
\item $R_n(G)$ is a subgroup of $G$?
\end{enumerate}
\end{qu}
 In the next sections   we shall discuss  Question \ref{q1} on various classes of groups $G$ and small positive integers $n$ and we also study many new questions extracted from it.

 The author has tried   the present  survey to be complete, but needless to say that  it does not contain  all  results on `the Engel structure' of groups.
Most of  results before 1970 was already surveyed in \cite[chapter 7 in Part II]{Robinson72-2} and  here we have only  sorted them  as
`left/right' results into separate sections. The latter reference is still more comprehensive than   ours for results before 1970.

  As   I
  believe  the following famous sentence of Hilbert \cite{Hilbert1902}, I have had  a tendency to write any question (not only ones  which are
  very difficult!) that I have encountered.

\emph{``As long as a branch of science offers an abundance of problems, so long
it is alive; a lack of problems foreshadows extinction or the cessation of
independent development.'' }

\section{Interaction of Right Engel Elements with  Left Engel Elements}
Baer \cite[Folgerung N and Folgerung A]{Baer57} proved that in groups with maximal condition on subgroups and in hyperabelian groups, a right
Engel element is a left Engel element.
The answer to the following question is still unknown.
\begin{qu}[Robinson] {\rm \cite[p. 370]{Robinson96}}\label{qu:Robinson=Baer}
Is it true that every right Engel element of any group is a left Engel element?
\end{qu}
  Heineken's result  \cite{Heineken60} gives the famous relation between left and right Engel elements of an arbitrary group;
 it can be read as follows:  {\em the inverse of a right Engel element is a left one}.
\begin{thm}[Heineken]{\rm\cite{Heineken60}}\label{t:He}
In any group $G$,
\begin{enumerate}
\item for any two elements $x,g\in G$ and all integers $n\geq 1$, $[x,_{n+1} g]=[g^{-x},_ng]^g$.
 \item $R(G)^{-1}\subseteq L(G)$.
\item $R_n(G)^{-1} \subseteq L_{n+1}(G)$.
\item $\overline{R}(G)^{-1} \subseteq \overline{L}(G)$.
\end{enumerate}
\end{thm}
\begin{proof}
All parts follows easily from 1. We may write
\begin{align*}
    [x,_{n+1} g]&=\big[[x,g],_ng\big]\\
    &=[x^{-1}g^{-1}xg,_ng]\\
    &=\big[\big(g^{-1}\big)^xg,_ng\big]\\
    &=\big[[\big(g^{-1}\big)^xg,g],_{n-1}g\big]\\
    &=\big[[\big(g^{-1}\big)^x,g]^g,_{n-1}g\big]\\
    &=\big[\big(g^{-1}\big)^x,_ng\big]^g\\
    \end{align*}
Now for instance, if $g\in G$ such that $g^{-1}\in R_n(G)$, then $\big(g^{-1}\big)^x\in R_n(G)$ for any $x\in G$ and so  $\big[\big(g^{-1}\big)^x,_n g\big]=1$ which implies that $[x,_{n+1}g]=1$,  by part 1.
\end{proof}
Apart from Heineken's result, we do not know  of other inclusions  holding  between Engel subsets in an arbitrary group.
We do not even know the answer of the following question.
\begin{qu}\label{qu:Ab}
\begin{enumerate}
\item For which integers $n\geq 1$,  $R_n(G)\subseteq L(G)$ for any group $G$?
\item Is it true that $\overline{R}(G)\subseteq L(G)$ for any group $G$?
\end{enumerate}
\end{qu}
The answer of Questions \ref{qu:Robinson=Baer} and \ref{qu:Ab} are known for many classes of groups (that we shall see them), but I do not know of even a ``knock" on the general case. Let us do that for part 1 of Question \ref{qu:Ab}. Suppose    $n>0$ is an integer  such that $R_n(G)\subseteq L(G)$ for any group $G$. We show that there is an integer $k>0$ depending only on $n$ such that $R_n(G)\subseteq L_k(G)$ for any group $G$. Let $\mathcal{G}_n$ be the group given by the following presentation $$\langle x,y \;|\; [x,_n X]=1 \;\;\text{for all}\;\; X\in\langle x,y\rangle \rangle.$$
Thus, it follows that  if  $G$ is a group generated by two elements $a$ and $b$ such that $a\in R_n(G)$,  then there is a group epimorphism $\varphi$ from $\mathcal{G}_n$ onto $G$ such that $x^\varphi=a$ and $y^\varphi=b$.  Since by assumption  $R_n(\mathcal{G}_n)\subseteq L(\mathcal{G}_n)$, we have $[y,_k x]=1$ for some $k$ depending only on $n$. This implies that $$1=[y,_k x]^\varphi=[y^\varphi,_k x^\varphi]=[b,_k a].$$
Therefore, to confirm validity of part 2 of Question \ref{qu:Ab}, one should find for any positive integer $n$, an integer $k$ such that $R_n(G)\subseteq L_k(G)$ for all groups $G$. So let us put forward the following question.
\begin{qu}
For which positive integers $n$, there exists a positive integer $k$ such that $R_n(G)\subseteq L_k(G)$ for all groups $G$?
\end{qu}
To refute  part 2 of Question \ref{qu:Ab} which is a question between bounded and unbounded Engel sets, one has to answer positively the following question on bounded Engel sets. Note that we do not claim answering this question is easier than else or vise versa!
\begin{qu}
Is there a positive integer $n$ such that for any given positive integer $k$ there is a group $G_k$ with $R_n(G_k)\nsubseteq L_k(G_k)$?
\end{qu}
What we know about other studies on  relations between left and right Engel elements are mostly on the `negative side'.
The following example of Macdonald  bounds Heineken's result ``$R_n^{-1}\subseteq L_{n+1}$''.
\begin{thm}[Macdonald] {\rm \cite{Macdonald}}
For any prime number $p$ and each multiple $n>2$ of $p$, there is a finite metabelian $p$-group $G$ containing an element $a\in R_n(G)$ such that $a\not\in L_n(G)$ and $a^{-1}\not\in L_n(G)$
\end{thm}
Macdonald's result was sharpened by L.-C. Kappe \cite{Kappe81}.
Newman and Nickel \cite{NewmanNickel94} showed that the situation may be  more bad: no non-trivial power of a right $n$-Engel element
can be a left $n$-Engel element.
\section{Four Engel Subsets and  Corresponding Subgroups}
In the most of  groups $G$ for which we know that  parts 1 to 4 of Question \ref{q1} are all true, the corresponding Engel subsets are equal to   the following subgroups, respectively:
the Hirsch-Plotkin  radical, the hypercenter, the Baer radical and the $\omega$-center of $G$.  In this section, we first shortly recall  definitions of these subgroups and then in the next section we collects known relations with the corresponding Engel subsets. There are also two other less famous  subgroups of an arbitrary group $G$  denoted by $\varrho(G)$ and $\overline{\varrho}(G)$ which are related to the right Engel elements.

Let $G$ be any group. We denote by   $Fitt(G)$  the Fitting subgroup of $G$ which is  the subgroup
generated by all normal nilpotent subgroups of $G$. By \cite[p. 100]{Fitting38}  the normal closure of each element of $Fitt(G)$ in $G$ is nilpotent.\\
 Let $HP(G)$ (called the Hirsch-Plotkin radical of $G$) be the subgroup generated by all normal locally nilpotent subgroups of $G$.
  Then by \cite{Hirsch55} or \cite{Plotkin54} $HP(G)$ is locally nilpotent.\\
We denote by $B(G)$ the set of elements $x\in G$ such that $\langle x\rangle$  is  a subnormal  subgroup in  $G$.
Then by \cite[\S 3, Satz 3]{Baer55} $B(G)$ (called the Baer radical of $G$) is a normal locally nilpotent subgroup of $G$ such that
every cyclic subgroup of $B(G)$ is subnormal in $G$.\\
We define inductively $\zeta_\alpha(G)$ (called $\alpha$-center of $G$) for each ordinal number $\alpha$. For $\alpha=0,1$, we have $\zeta_0(G)=1$
and $\zeta_1(G)=Z(G)$ the center of $G$. Now suppose $\zeta_\beta(G)$ has been defined for any ordinal $\beta<\alpha$.
If $\alpha$ is not a limit ordinal (i.e., $\alpha=\alpha'+1$ for some ordinal $\alpha'<\alpha$), we define $\zeta_{\alpha}(G)$ to be
such that $$Z\big(\frac{G}{\zeta_{\alpha'}(G)}\big)=\frac{\zeta_{\alpha}(G)}{\zeta_{\alpha'}(G)},$$    and if $\alpha$ is a limit ordinal,
we define $\zeta_\alpha(G)=\displaystyle\bigcup_{\beta<\alpha}\zeta_\beta(G)$.\\
We denote by $\omega$ the ordinal of natural numbers $\mathbb{N}$ with the usual order $<$. The ordinal $\omega$ is the first infinite ordinal and it
is a limit one. It follows that  $$\zeta_\omega(G)=\bigcup_{\beta<\omega}\zeta_{\beta}(G).$$
Since every ordinal $\beta<\omega$ is a finite one,  every such   $\beta$ can be considered as a non-negative integer. Thus we have
$$Z\big(\frac{G}{\zeta_{i}(G)}\big)=\frac{\zeta_{i+1}(G)}{\zeta_{i}(G)} \;\; {\rm for ~ each ~ integer~} i\geq 0.$$
Since the cardinal of a group $G$ cannot be exceeded, there is an ordinal $\beta$ such that $\zeta_\lambda(G)=\zeta_{\beta}(G)$ for all ordinal $\lambda\geq \beta$.
For such an ordinal $\beta$, we call $\zeta_{\beta}(G)$ the hypercenter of $G$ and it will be denoted by $\overline{\zeta}(G)$.\\
We denote by $Gr(G)$ the set of elements $x\in G$ such that $\langle x\rangle$  is  an ascendant  subgroup in  $G$.
Then by \cite[Theorem 2]{Gruenberg59} $Gr(G)$ (called the Gruenberg radical of $G$) is a normal locally nilpotent subgroup of $G$
such that every cyclic subgroup of $Gr(G)$ is ascendant in $G$.
A group $G$ is called a Fitting, hypercentral, Baer or Gruenberg group if $Fitt(G)=G$, $\overline{\zeta}(G)=G$, $B(G)=G$ and $Gr(G)=G$, respectively.
Note that $$Fitt(G)\leq B(G)\leq Gr(G)\leq HP(G)$$ for any group $G$.

For a group $G$, following Gruenberg \cite[p. 159]{Gruenberg59} we define $\varrho(G)$ to be the set of all elements
$a$ of $G$ such that $\langle x\rangle$ is an ascendant  subgroup of $\langle x\rangle\langle a\rangle^G$ for every $x\in G$.
Similarly, $\overline{\varrho}(G)$ is defined to be the set of all elements $a\in G$  for which there is a positive integer $n=n(a)$
such that $\langle x\rangle$ is a subnormal  subgroup  in $\langle x\rangle\langle a\rangle^G$ of defect at most $n$  for every $x\in G$.\\
By \cite[Theorem 3]{Gruenberg59} $\varrho(G)$ and $\overline{\varrho}(G)$ are characteristic subgroups of $G$ satisfying
$$\zeta_\omega(G)\leq \overline{\varrho}(G) \;\;\text{and}\;\; \overline{\zeta}(G)\leq \varrho(G).$$
In addition, $\varrho(G)\leq Gr(G)$ and $\overline{\varrho}(G)\leq B(G)$.

\subsection{Left Engel Elements}
In this section we will deal with left Engel elements.
\begin{prop}
For any group $G$, $HP(G)\subseteq L(G)$.
\end{prop}
\begin{proof}
 Let $g\in HP(G)$ and $x\in G$. Then $\langle g^{-x},g\rangle\leq HP(G)$, as $HP(G)$ is normal in $G$. Thus $\langle g^{-x},g \rangle$ is nilpotent of class at most $n\geq 1$, say. By  Theorem \ref{t:He}, we have $[x,_{n+1} g]=[g^{-x},_{n}g]^g$. As $[g^{-x},_n g]=1$, we have $[x,_{n+1} g]=1$. This implies that
$g\in L(G)$.
\end{proof}
 Let $p$ be a prime number or zero and let  $d\geq 2$ be any integer.  Golod \cite{Golod} has constructed a non-nilpotent (infinite) $d$-generated, residually finite, $p$-group (torsion-free group whenever $p=0$)  $G_d(p)$ in which every $d-1$ generated subgroup is a nilpotent group. Hence for any $d\geq 3$ every two generated subgroup of $G_d(p)$ is nilpotent so that the Golod group $G_d(p)$ is an Engel group, that is, $G_d(p)=R(G_d(p))=L(G_d(p))$.
Therefore for an arbitrary group $G$, it is not necessary to have  $HP(G)=L(G)$ as
$L(G_3(p))\not=HP(G_3(p))$.
\begin{prop}
For any group $G$, $B(G)\subseteq \overline{L}(G)$.
\end{prop}
\begin{proof}
  Let $g\in B(G)$. Thus $\langle g\rangle$ is subnormal in $G$ of defect $k$, say. Therefore $\langle g\rangle[G,_k\langle g\rangle]=\langle g\rangle$. It follows that $[G,_k \langle g\rangle]\leq \langle g\rangle$ and so $[G,_{k+1} \langle g\rangle]=1$. Hence $[x,_{k+1} g]=1$ for all $x\in G$ and $g\in \overline{L}(G)$. \end{proof}
  So, elements of the Hirsch-Plotkin radical (the Baer radical, resp.) of a group are potential examples of (bounded, resp.) left Engel elements of a group.
 An element of order $2$ (if exists) in a group, under some extra conditions, can also belong to the set of (bounded) left Engel elements.
\begin{prop}\label{inv}
Let $x$ be an element of any group  $G$ such that $x^2=1$.
\begin{enumerate}
 \item $[g,_n x]=[g,x]^{(-2)^{n-1}}$ for all $g\in G$ and all integers $n\geq 1$.
 \item If every commutator of weight $2$ containing $x$ is a $2$-element, then $x\in L(G)$. In particular, if $G'$ is a $2$-group, then $x\in L(G)$.
 \item  If every commutator of weight $2$ containing $x$ is of order dividing $2^n$, then $x\in L_{n+1}(G)$. In particular, if $G'$ is of exponent dividing $2^n$, then $x\in L_{n+1}(G)$.
 \end{enumerate}
\end{prop}
\begin{proof}
It is enough to show part 1. We argue by induction on $n$. It is clear, if $n=1$. We have
\begin{align*}
[g,_{n+1} x]&=[g,_nx]^{-1}[g,_nx]^x \\
&=\big([g,x]^{(-2)^{n-1}}\big)^{-1}\big([g,x]^{(-2)^{n-1}}\big)^x\\
&=[g,x]^{(-1)\cdot(-2)^{n-1}}\big([g,x]^x\big)^{(-2)^{n-1}} \\
&=[g,x]^{(-1)\cdot(-2)^{n-1}} [g,x]^{(-1)\cdot(-2)^{n-1}} \;\; {\rm since}\; x^2=1, \; [g,x]^{x}=[g,x]^{-1}\\
&=[g,x]^{(-2)^{n}}.
\end{align*}
This completes the proof.
\end{proof}
  The above phenomenon, something like part 2 of Proposition \ref{inv},   may not be true   for  elements of other prime orders.
  For, if $p\geq 4381$ is a prime,  then the free $2$-generated Burnside group $\mathcal{B}=B(2,p)$  of exponent $p$,  by a deep result of Adjan and Novikov \cite{AdjanNovikov68} is infinite and every abelian subgroup of $B$ is finite. Now, if a non-trivial element $a\in \mathcal{B}$  were in $L(\mathcal{B})$, then $A=\langle a\rangle^\mathcal{B}$ would be nilpotent by Theorem \ref{t:max-L}. If $A$ is infinite, then as it is nilpotent, it contains an infinite abelian subgroup which is not possible and if $A$ is finite, then $a$ has finitely many conjugates in $\mathcal{B}$ and in particular the centralizer $C_\mathcal{B}(a)$ is infinite, which is again impossible as the centralizer of every non-trivial element in $\mathcal{B}$ is finite by \cite{AdjanNovikov68}.

The following result  was announced by Bludov in \cite{Bludov2005}.
\begin{thm}[Bludov] {\rm \cite{Bludov2005}}
There exist groups in which a product of left Engel elements is not
necessarily a left Engel element.
\end{thm}
This refutes part 1 of Question \ref{q1} for an arbitrary group, i.e., the set of left Engel elements is not in general a subgroup. He constructed a non Engel group generated by left Engel elements. In
particular, he shows  a non left Engel element which is a product of two left Engel elements.
His example is based on infinite $2$-groups  constructed by Grigorchuk \cite{Grigorchuk}. Note that part 3 of Question \ref{q1} is still open, i.e., we do not know whether the set of bounded left Engel elements is a subgroup or not. To end this section we prove the following result.
\begin{prop}
 At least one of the following  happens.
 \begin{enumerate}
 \item The free $2$-generated Burnside group of exponent $2^{48}$ is an $k$-Engel group for some integer $k$.
\item There exists  a group $G$ of exponent $2^{48}$ generated by four involutions which is not an Engel group.
 \end{enumerate}
 \end{prop}
\begin{proof}
Let $n= 2^{48}$ and $B(X,n)$ be the free Burnside  group on the set $X=\{x_i \;|\; i\in \mathbb{N}\}$ of the Burnside variety of exponent $n$ defined by the law
$x^n=1$. Lemma 6 of \cite{IvanovOlshanskii97} states that  the subgroup $\langle x_{2k-1}^{n/2}x_{2k}^{n/2}\;|\; k=1,2,\dots\rangle$ of $B(X,n)$ is isomorphic  to $B(X,n)$ under the map $x_{2k-1}^{n/2}x_{2k}^{n/2}\rightarrow x_k$, $k=1,2,\dots$. Therefore  the subgroup
$\mathcal{G}:=\langle x_1^{n/2},x_2^{n/2},x_3^{n/2},x_4^{n/2}\rangle$ is generated by four elements of order $2$, contains the subgroup $\mathcal{H}=\langle x_1^{n/2}x_2^{n/2},x_3^{n/2}x_4^{n/2} \rangle$ isomorphic to the free $2$-generator Burnside group $B(2,n)$ of exponent $n$.  It follows from Proposition \ref{inv} that the group $\mathcal{G}$ can be generated by four left $49$-Engel elements of $\mathcal{G}$. Thus $$\mathcal{G}=\langle L_{49}\big(\mathcal{G}\big)\rangle=\langle L\big(\mathcal{G}\big)\rangle=\langle \overline{L}\big(\mathcal{G}\big)\rangle.$$

Suppose, if possible, $\mathcal{G}$ is an Engel group. Then $\mathcal{H}$ is also an Engel group. Let $Z$ and $Y$ be two free generators of $\mathcal{H}$. Thus $[Z,_k Y]=1$ for some integer $k\geq 1$. Since $\mathcal{H}$ is the free 2-generator Burnside group of exponent $n$, we have that every group of exponent $n$ is a $k$-Engel group. Therefore, $\mathcal{G}$ is an infinite finitely generated $k$-Engel group of exponent $n$, as $\mathcal{H}$ is infinite by a celebrated result of Ivanov \cite{Ivanov94}. This completes the proof.
 \end{proof}
We believe that the group $\mathcal{H}$ cannot be an Engel group, but we are unable to prove it.
\subsubsection{Left Engel Elements in Generalized Soluble  and Linear Groups}
In this subsection we collect main results on Engel structure left Engel elements in generalized soluble  and linear groups.

The papers   \cite{Gruenberg59} and \cite{Gruenberg61} by Gruenberg are essential  to anyone who wants to know the Engel structure of soluble groups.
In particular, all four Engel subsets are subgroups in any soluble group.
\begin{thm}[Gruenberg]{\rm \cite[Proposition 3, Theorem 4]{Gruenberg59}} Let $G$ be a soluble group. Then  $L(G)=HP(G)$
and $\overline{L}(G)=B(G)$.
\end{thm}
A group is called radical if it has an ascending series with locally nilpotent factors. Define the upper Hirsch-Plotkin series of a group to be the ascending series $1=R_0\leq R_1\leq \cdots $ in which $R_{\alpha+1}/R_{\alpha}=HP\big(G/R_{\alpha}\big)$ and $R_\lambda=\bigcup_{\alpha<\lambda}R_{\alpha}$ for limit ordinals $\lambda$. It can be proved that radical groups are precisely those groups which coincide with a term of their upper Hirsch-Plotkin series.
\begin{thm}[Plotkin]{\rm \cite[Theorem 9]{Plotkin55}}\\
Let $G$ be a radical group. Then $L(G)=HP(G)$.
\end{thm}
\begin{qu}[Robinson]{\rm \cite[p. 63 of Part II]{Robinson72-2}}
Let $G$ be a radical group. Is it true that $\overline{L}(G)=B(G)$?
\end{qu}
\begin{thm}[Gruenberg] {\rm \cite{Gruenberg66}} Let $R$ be a commutative Noetherian ring with identity and $G$ be a group of $R$-automorphisms of a finitely
generated $R$-module. Then $L(G)=HP(G)$ and $\overline{L}(G)=B(G)$.
\end{thm}
To state some results on ceratin soluble groups we need the following definitions.
Let $\mathfrak{A}_0$ be the class of abelian groups of finite torsion-free rank and finite
$p$-rank for every prime $p$; $\mathfrak{A}_1$ be the class of abelian groups $A$ of finite pr\"ufer rank  such that $A$ contains
only a finite number of elements of prime order; $\mathfrak{A}_2$ be the class of abelian groups which have a series of finite length
each of whose factors satisfies either the maximal or the minimal
condition for subgroups; and let $\mathfrak{S}_i$ be  the class of all poly $\mathfrak{A}_i$-groups.
The class of all $\mathfrak{S}_0$-groups in which
 the product of all  periodic normal
subgroups is finite will be denoted by $\mathfrak{S}_t$.
\begin{thm}[Gruenberg]{\rm \cite[Theorem 4]{Gruenberg66} and \cite[Theorem 1.3, Proposition 1.1, Lemma 2.2, Proposition 6.1]{Gruenberg61}}
 If $G$ is an $\mathfrak{S}_0$-group, then
 \begin{enumerate}
\item $HP(G)=Gr(G)=L(G)$ is hypercentral;
\item $\overline{\zeta}(G)=\varrho(G)=R(G)=\zeta_{\alpha}(G)$, where  $\alpha\leq n\omega$ for some
positive integer $n$;
\item  $Fitt(G)$ need not be nilpotent.
\end{enumerate}
If $G$ is  an $\mathfrak{S}_1$-group, then
\begin{enumerate}
\item $Fitt(G)=B(G)=L(G)$ is nilpotent;
\item $\zeta_\omega(G)=\overline{\varrho}(G)=R(G)$;
 \item $HP(G)$ need not be nilpotent and $\overline{\zeta}(G)$ may not equal to $\zeta_n(G)$ for some positive integer $n$,
even if $G$ satisfies the minimal condition — and therefore is an $\mathfrak{S}_2$-group.
\end{enumerate}
If G is  an $\mathfrak{S}_t$-group, then
 $HP(G)=Fitt(G)$ and $\overline{\zeta}(G)=\zeta_k(G)$ for some positive integer $k$.
\end{thm}
\begin{thm}[Wehrfritz]{\rm \cite[Theorem E2]{Wehrfritz70}}
Let $G$ be a group of automorphisms of the finitely generated
$\mathfrak{S}_0$-group $A$. Then
\begin{enumerate}
\item[(i)] $HP(G)=Gr(G)=L(G)$ is hypercentral;
\item[(ii)] $\overline{\zeta}(G)=\varrho(G)=R(G)=\zeta_{\alpha}(G)$, where $\alpha\leq n\omega$ for some
positive integer $n$;
\item[(iii)] $HP(G)=B(G)=\overline{L}(G)$ is nilpotent;
\item[(iv)] $\zeta_\omega(G)=\overline{\varrho}(G)=\overline{R}(G)$.\\
 If in addition $A$ is an $\mathfrak{S}_t$-group, then
\item[(v)] $HP(G)=Fitt(G)$ and $\overline{\zeta}(G)=\zeta_k(G)$ for some positive integer $k$.
\end{enumerate}
\end{thm}
\begin{thm}[Gruenberg] {\rm \cite{Gruenberg66}} Let $R$ be a commutative Noetherian ring with identity and $G$ be a group of $R$-automorphisms of a finitely
generated $R$-module. Then $L(G)=HP(G)$ and $\overline{L}(G)=B(G)$.
\end{thm}
Let us finish this section with some results on the structure of generalized linear groups given by Wehrfritz. Let $R$
denote a commutative ring with identity and $M$ an $R$-module. Let $G$ be a group of
finitary  automorphisms of $M$ over $R$; that is,
\begin{align*} G &\leq FAut_R M=\{g \in  Aut_RM \;:\; M(g-1) \;\text{is}\; R-\text{Noetherian}\}\\
&\leq Aut_RM.
\end{align*}
\begin{thm}[Wehrfritz]{\rm \cite[4.4]{Wehrfritz2000}}
Let $G$ be a group of finitary automorphisms of  a module  over
a commutative ring  with identity. Then $L(G)=HP(G)=Gr(G)$ and $\overline{L}(G)=B(G)$.
\end{thm}
Wehrfritz has also studied the Engel structure of finitary skew linear groups \cite{Wehrfritz92}.
 \subsubsection{Left Engel Elements in Groups Satisfying Certain Min or Max Conditions}
The famous structure result for left Engel elements  is due to Baer \cite[p.257]{Baer57}, where he proved that  a left Engel element defined by right-normed commutators of a group satisfying maximal condition on subgroups belongs to the Hirsch-Plotkin radical. Therefore in groups satisfying maximal conditions, the set of left Engel elements defined by right-normed commutators is a subgroup and so it coincides with the one defined by left-normed commutators. Hence, Baer's result is also valid for left Engel element defined by left-normed commutators.
\begin{thm}[Plotkin]{\rm \cite{Plotkin58}}\label{Plotkin58}\label{t:max-L}
Let $G$ be a group which satisfies the maximal condition on its abelian subgroups. Then $L(G)=\overline{L}(G)=HP(G)$ which is nilpotent.
\end{thm}
Theorem \ref{t:max-L} follows from the following key result due to Plotkin \cite{Plotkin58}.
\begin{thm}[Plotkin] {\rm \cite[Lemma 2]{Plotkin58}}\label{t:Plotkin58}
Let $G$ be an arbitrary group and $g\in L(G)$. Then there exists a sequence of subgroups $$H_1\leq H_2\leq\cdots\leq H_n\leq\cdots$$ in $G$ satisfying the following conditions:
\begin{enumerate}
\item $H_i$ is nilpotent for all integers $i\geq 1$,
\item $H_1=\langle g\rangle$ and for each $i\geq 2$, $H_i=\langle H_{i-1},g^{h_i}\rangle$ for some $h_i\in G$,
\item $H_i$ is normal in $H_{i+1}$ for all $i\geq 1$.
\item there is an integer $n\geq 1$ such that $H_{n+1}=H_n$  if and only if $H_n$ is a normal subgroup of $G$.
\end{enumerate}
\end{thm}
Theorem \ref{t:Plotkin58} follows from the following important result of \cite[Lemma 1]{Plotkin58}. Note that if a group satisfies maximal condition on its abelian subgroups, then by \cite[Theorem 3.31]{Robinson72-2} it also satisfies maximal condition on its nilpotent subgroups.
\begin{thm}[Plotkin]{\rm \cite[Lemma 1]{Plotkin58}}\label{thm:Plotkin58Lemma1}
Let $H$ be a nilpotent subgroup of any group $G$ such that  $H=\langle H\cap L(G)\rangle$. If $H$ is not normal in $G$, then there is an element $x\in N_G(H)\setminus H$ which is conjugate to some element of $H\cap L(G)$.
\end{thm}
To taste a little of the proof of Theorem \ref{thm:Plotkin58Lemma1}, let us  treat the case in which $H$ is  finite cyclic  generated by $g\in L(G)$. Since $H$ is not normal, there is an element $x\in G$ such that $g^x\not\in H$. Thus $[x,g]\not\in H$ and since $g\in L(G)$, there exists an integer $n\geq 2$ such that $[x,_n g]=1$. It follows that there is a positive integer $k$ such that $[x,_k g]\not\in H$ but $[x,_{k+1} g]\in H$. Since $[x,_{k+1} g]=g^{-[x,_k g]}g\in H$ and $g\in H$, we have that $g^{[x,_k g]}\in H$. This implies that $H^{[x,_k g]}\leq H$ and since $H$ is finite, $[x,_k g]\in N_G(H)$.
\begin{co}\label{co:lfl}
Let $G$ be a locally finite group. Then $L(G)=HP(G)$.
\end{co}
\begin{proof}
Let $x\in L(G)$ and $g_1,\dots,g_n \in G$. Then $H=\langle x^{g_1},\dots,x^{g_n}\rangle$ is a finite group generated by left Engel elements and so by Theorem \ref{Plotkin58}, $H$ is  nilpotent. This implies that $\langle x \rangle^G$ is locally nilpotent and so $x\in HP(G)$, as required.
\end{proof}
A group is said to satisfy Max locally whenever every finitely generated subgroup  satisfies the maximal condition on its subgroups.
\begin{thm}[Plotkin] {\rm \cite{Plotkin58}}
Let $G$ be a group having an ascending series whose factors satisfy Max locally. Then $L(G)=HP(G)$.
\end{thm}
\begin{thm}[Held] {\rm \cite{Held66}}\label{thm:Held66}
Let $G$ be a group satisfying minimal condition on its abelian subgroups. Then $\overline{L}(G)=Fitt(G)$.
\end{thm}
As far as we know the following result is still unpublished.
\begin{thm}[Martin]{\rm \cite[p. 56 of Part II]{Robinson72-2}}
Let $G$ be a group  satisfying minimal condition on its abelian subgroups. Then $L(G)=HP(G)$.
\end{thm}
A group $G$ is called an $\mathfrak{M}_c$-group or said to satisfy $\mathfrak{M}_c$ (the minimal condition on centralizers) whenever for the centralizer $C_G(X)$ of any set of elements $X$ of $G$, there is a finite subset $X_0$ of $X$ such that $C_G(X)=C_G(X_0)$.
\begin{thm}[Wagner]{\rm \cite[Corollary 2.5]{Wagner99}}\\
Let $G$ be an $\mathfrak{M}_c$-group. Then $\overline{L}(G)=Fitt(G)$.
\end{thm}
\begin{thm}
Let $G$ be an $\mathfrak{M}_c$-group. Then every left Engel element of prime power order of $G$ lies in the Hirsch-Plotkin radical of $G$.
\end{thm}
\begin{proof}
Let $x\in L(G)$ be a $p$-element for some prime $p$. Then the set  of all conjugates of $x$ in $G$ is a $G$-invariant subset of $p$-elements in which every pair of elements satisfies some Engel identity (in the sense of \cite[Definition 1.2]{Wagner99}). Now \cite[Corollary 2.2]{Wagner99} implies that $\langle x\rangle^G$ is a locally finite $p$-group so that $x\in HP(G)$. This completes the proof.
\end{proof}
\begin{qu}
Let $G$ be an $\mathfrak{M}_c$-group. Is it true that $L(G)=HP(G)$?
\end{qu}
A group $G$ is said to have finite (Pr\"ufer) rank  if there is an integer $r>0$ such that every finitely generated subgroup of $G$ can be generated by $r$ elements.
A group $G$ is said to have finite abelian subgroup rank if every abelian subgroup of $G$ has finite rank.
\begin{qu}
 Is $L(G)$ a subgroup for groups $G$ with finite rank? Is $L(G)$  a subgroup for groups $G$ with finite abelian subgroup rank?
\end{qu}
\subsubsection{Left $k$-Engel Elements}
In this subsection, we deal with left $k$-Engel elements, specially for small values of $k$.\\
Left 1-Engel elements are precisely  elements of the center. Left $2$-Engel elements can  easily be characterized
\begin{prop}\label{prop:l2}
\begin{enumerate}
\item For any group $G$,  $L_2(G)=\big\{ x\in G\;|\; \langle x\rangle^G \;\text{is abelian}\big\}$. In particular, $L_2(G)\subseteq Fitt(G)$
\item There is a group $K$ in which $L_2(K)$ is not a subgroup.
\end{enumerate}
\end{prop}
\begin{proof}
1.  The proof  follows from the fact that for any elements $a,b,x\in G$, we have:
\begin{align*}
[ab^{-1},_2x]=1 &\Leftrightarrow [x^{-ab^{-1}}x,x]=1\\ &\Leftrightarrow [x^{-ab^{-1}},x]=1 \\ &\Leftrightarrow[x^{ab^{-1}},x]=1 \\ &\Leftrightarrow [x^a,x^b]=1.
 \end{align*}
 2.  Take $K$ to be the standard wreath product of a group of order $2$ and with an elementary abelian group of order $4$. The group $K$ is generated by left $2$-Engel elements but $K\not=L_2(K)$. This completes the proof of part 2.
\end{proof}
\begin{prop}
Let $A$ be any group of exponent $2^k$ for some integer $k\geq 1$ and $\langle x\rangle$ and $\langle y\rangle$ be cyclic groups of order $2$. Let $G$ be the  standard wreath product $A\wr \big (\langle x\rangle \times \langle y\rangle\big)$. Then
\begin{enumerate}
\item $x,y,xy\in L_{k+1}(G)\setminus L_{k}(G)$.
\item $ax \not\in L_{k+1}(G)$ for all $1\not=a\in A$.
\end{enumerate}
In particular, for any integer $n\geq 2$, there exists a group $G$ containing two elements $a,b\in L_n(G)$ such that $ab\not\in L_{n}(G)$.
\end{prop}
Let $x$ be a bounded left (right, resp.) Engel element of a group $G$. The left (right, resp.) Engel length of $x$ is defined to be the least non-negative  integer $n$ such that $x\in L_n(G)$ ($x\in R_n(G)$, resp.) and it is denoted by $\ell^l_G(x)$ ($\ell^r_G(x)$, resp.).
 Roman'kov \cite[Question 11.88]{Kourovka} asked  whether for any group $G$ there exists a linear
(polynomial) function $\phi(x, y)$ such that $\ell_G^l(uv)\leq \phi(\ell_G^l(u),\ell_G^l(v))$ for elements $u,v\in\overline{L}(G)$.
Dolbak \cite{Dolbak} answered negatively the  question of Roman'kov \cite[Question 11.88]{Kourovka}.
We propose the following problem.
\begin{prob} \label{prob:ll}Let $G$ be an arbitrary group. Find all  pairs $(n,m)$ of positive integers such that,
   $xy\in \overline{L}(G)$ whenever $x\in L_n(G)$ and $y\in L_m(G)$.
\end{prob}
We now shortly show that for every integer $m>0$, all pairs $(2,m)$  are of the solutions of part 1 of Problem \ref{prob:ll}.
\begin{prop}
Let $G$ be any group and $a\in L_2(G)$ and $b\in L_n(G)$ for some $n\geq 1$. Then both $ab$ and $ba$ are in  $L_{2n}(G)$.
\end{prop}
\begin{proof}
Let $g\in G$ and $X=\langle a\rangle^G$. Then $$[g,_{n}ab]X=[gX,_{n}aXbX]=[g,_n b]X=X,$$ where the last equality holds as $b\in L_n(G)$. Therefore $[g,_{n}ab]\in X$. So we have
\begin{align*}
[g,_{2n}ab]&=[[g,_n ab],_n ab]\\
&=[[g,_n ab],_n b] \;\;\text{Since}\; [g,_n ab], a\in X \;\text{and}\; X \;\text{is abelian normal in}\; G\\
&=1 \;\;\text{Since}\; b\in L_n(G).
\end{align*}
This proves that $ab\in L_{2n}(G)$. Since $L_{2n}(G)$ is closed under conjugation, $(ab)^a=ba$ is also in $L_{2n}(G)$. This completes the proof.
\end{proof}
Let us ask the following question that we suspect it to be true.
\begin{qu}
Is it true that the product of every two left $3$-Engel element is a left Engel element?
\end{qu}

The following question is arisen by part 1 of Proposition \ref{prop:l2}. Is it true that every bounded left Engel element is in the Hirsch-Plotkin radical? In general for an arbitrary group $K$ it is not necessary that
$L_n(K)\subseteq HP(K)$.  Suppose, for a contradiction, that $L_n(K)\subseteq HP(K)$ for all $n$ and all groups $K$. By a deep result of Ivanov \cite{Ivanov94}, there is a finitely generated infinite group
$M$ of exponent $2^k$ for some positive integer $k$. Suppose that $k$ is the least integer with
this property, so every finitely generated group of exponent dividing $2^{k-1}$ is finite. By Proposition \ref{inv}  every element of order 2 in
$M$ belongs to $L_{k+1}(M)$. So by hypothesis, $M=HP(M)$ is of exponent dividing $2^{k-1}$ and so
it is finite. Since $M$ is finitely generated, $HP(M)$ so is. But this yields that $HP(M)$
is a periodic finitely generated nilpotent group and so it is finite. It follows that $M$ is
finite, a contradiction.  This argument can be found in \cite{Abdollahi2004}.\\
Hence the following question  naturally arises.
\begin{qu}\label{qu:lnhp}
 What is the least positive integer
$n$ for which  there is a group $G$ with  $L_n(G)\nsubseteq HP(G)$?
\end{qu}
If one uses Lysenkov's result \cite{Lysenkov96} instead of Ivanov's one \cite{Ivanov94} in the above argument, we find that the requested integer $n$ in Question \ref{qu:lnhp} is less than or equal to $13$.  To investigate Question \ref{qu:lnhp} one should first study the case
$n = 3$ which  was already started in \cite{Abdollahi2004}.
\begin{prop}[Abdollahi] {\rm \cite[Corollary 2.2]{Abdollahi2004}}\label{co:Abdollahi2004}
For an arbitrary group $G$,
$$L_3(G) = \{x\in G \;|\; \langle x, x^y\rangle \;\text{is nilpotent of class at most}\; 2 \;\text{for all}\; y\in G\}.$$
In particular, every power of a left $3$-Engel element is also a left $3$-Engel element.
\end{prop}
\begin{thm}[Abdollahi] {\rm \cite[Theorem 1.1]{Abdollahi2004}}\label{thm:Abdollahi2004}
Let  $p$ be any prime number and $G$ be a group. If   $x\in L_3(G)$ and $x^{p^n}=1$ for some integer $n>1$, then $\langle x^p\rangle^G$ is soluble of derived length at most $n-1$ and $x^p\in B(G)$. In particular, $\langle x^p\rangle^G$ is locally nilpotent.
\end{thm}
Theorem \ref{thm:Abdollahi2004} reduces the verification of the question whether any left $3$-Engel element of prime power order lies in the Hirsch-plotkin radical to the following.
\begin{qu}\label{qu:l3-p}
Let $G$ be a group and $x\in L_3(G)$ of prime   order $p$. Is it true that $x\in HP(G)$?
\end{qu}
The positive answer of Question \ref{qu:l3-p} for the case $p=2$   gives a new proof for the local finiteness of groups of exponent $4$.

Two left $3$-Engel elements generate a nilpotent group of class at  most $4$.
\begin{thm}[Abdollahi] {\rm \cite[Theorem 1.2]{Abdollahi2004}}\label{thm:Abdollahi2004-1.2}
Let G be any group and $a, b\in L_3(G)$. Then $\langle a, b\rangle$ is nilpotent of class
at most $4$.
\end{thm}
NQ package \cite{Nickel98} can show  $4$ is the best bound in Theorem \ref{thm:Abdollahi2004-1.2}.
\begin{qu}[Abdollahi] {\rm \cite[Question]{Abdollahi2004}} Is there a function $f:\mathbb{N}\rightarrow\mathbb{N}$ such that every nilpotent group generated by
$d$ left $3$-Engel elements is nilpotent of class at most $f(d)$?
\end{qu}
 In particular, whether the number $n$ of Question \ref{qu:lnhp} is greater than $4$ or not is of special interest. Indeed, if it were greater than $4$, then every group of exponent $8$ would be locally finite. Here are some thought on the left $4$-Engel elements of a group.
 \begin{thm}[Abdollahi \& Khosravi]{\rm \cite[Theorem 1.5]{AbdollahiKhosravi1}}
Let $G$ be an arbitrary group and both $a$ and $a^{-1}$ belong to $L_4(G)$. Then  $\langle a,a^b\rangle$ is nilpotent of class at most
$4$ for all $b\in G$.
\end{thm}
\begin{thm}[Abdollahi \& Khosravi]{\rm \cite[Theorem 1.6]{AbdollahiKhosravi1}}
Let $G$ be a group. If both $a$ and $a^{-1}$ belong to $L_4(G)$ and are of  $p$-power  order
for some prime $p$, then
\begin{enumerate}
    \item If $p=2$ then $a^4\in B(G)$.
    \item If $p$ is an odd prime, then $a^p\in B(G)$.
\end{enumerate}
\end{thm}
\subsection{Right Engel Elements}
In this section we discuss on right Engel elements of groups.
\begin{prop}
Let $G$ be any group and $g\in \zeta_{\omega}(G)$. Then there is an integer $n>0$ such that $\langle g,x\rangle$ is nilpotent of class at most $n$ for all $x\in G$. In particular, $\zeta_\omega(G)\subseteq \overline{R}(G)$.
\end{prop}
\begin{proof}
Since $g\in \zeta_{\omega}(G)$, there exists an integer $n>0$ such that $g\in \zeta_{n}(G)$.  Now consider the factor group  $\frac{\langle g,x\rangle\zeta_n(G)}{\zeta_n(G)}$. As $g\in \zeta_n(G)$, $\frac{\langle g,x\rangle}{\langle g,x\rangle\cap \zeta_n(G)}$ is cyclic. Since  $\langle g,x\rangle\cap \zeta_n(G)$ is contained in $\zeta_n\big(\langle g,x\rangle\big)$, we have $\frac{\langle g,x\rangle}{\zeta_n\big( \langle g,x\rangle\big)}$ is cyclic. It follows that $\langle g,x\rangle=\zeta_n\big( \langle g,x\rangle\big)$ and so  $\langle g,x\rangle$ is nilpotent of class at most $n$. This easily implies that $[g,_n x]=1$ and so $g\in \overline{R}(G)$.
\end{proof}
\begin{prop}[Gruenberg]{\rm\cite[Theorem 3]{Gruenberg59}}
Let $G$ be any group. Then $\varrho(G)\subseteq R(G)$ and $\overline{\varrho}(G)\subseteq \overline{R}(G)$.
\end{prop}
\begin{proof}
   Let $a\in \varrho(G)$ and $x\in G$. Then $\langle x\rangle$ is ascendant in $\langle x\rangle\langle a\rangle^G$. Therefore $[a,_n x]=1$ for some integer $n$; for otherwise, by examining those terms of ascending series between $\langle x\rangle$ and $\langle x\rangle\langle a\rangle^G$ that contain $[a,_n x]$, we should be able to find a set of ordinals without a first element. Hence $\varrho(G)\subseteq R(G)$. It is equally easy to see that $\overline{\varrho}(G)\subseteq \overline{R}(G)$.
\end{proof}
\subsubsection{Right Engel Elements in Generalized Soluble and Linear Groups}
\begin{thm}[Gruenberg] {\rm \cite[Theorem 4]{Gruenberg59}}
Let $G$ be a soluble group. Then  $R(G)=\varrho(G)$ and $\overline{R}(G)=\overline{\varrho}(G)$.
\end{thm}
The standard
wreath product of a cyclic group of prime order $p$ by an elementary abelian $p$-group
of infinite rank is an example of a soluble group with trivial hypercenter, but in which
all elements are right $p+1$-Engel. Replacing the elementary abelian $p$-group
in the wreath product by a quasicyclic $p$-group gives a soluble  group $G$ with
$\overline{\zeta}(G)=\overline{R}(G)=1$ but satisfying $G=R(G)$. Gruenberg \cite[Theorem 1.7]{Gruenberg59} showed that $\overline{R}(G)=\zeta_\omega(G)$ when $G$ is a finitely
generated  soluble group. Brookes \cite{Brookes86} proved that all four subsets are the same in a finitely generated soluble group.
\begin{thm}[Brookes] {\rm \cite[Theorem A]{Brookes86}}\label{thm:Brookes} Let $G$ be a finitely generated soluble group. Then
$$R(G)= \overline{\zeta}(G)=\overline{R}(G)= \zeta_\omega(G).$$
\end{thm}
The proof of Theorem \ref{thm:Brookes} follows from a result on constrained modules over integer group rings: For a group $G$ and a
commutative Noetherian ring $R$, an $RG$-module $M$ is said to be constrained if for all $m\in M$ and $g\in G$ the $R$-module $mR\langle g\rangle$ is finitely generated as an $R$-module. Let us explain how such modules are appeared in the study of right Engel elements. Let $G$ be any group, $H$ a normal subgroup of $G$ and $K$ a normal subgroup of $H$ such that $M=H/K$ is abelian and suppose further that $H=\langle H\cap R(G)\rangle$. Then $G$ acts by conjugation on $M$ as a group and so $M$ can be considered as a $\mathbb{Z}G$-module. Now we prove that $M$ is a restrained $\mathbb{Z}G$-module. We need the following technical and useful result to show our claim as well as in the sequel.
\begin{lemma}\label{lem:x^y}
Let $x,y$ be elements of a group $G$.
\begin{enumerate}
\item For each integer $k\geq 0$, there exist elements $g_k(x,y),f_k(x,y),h_k(x,y)\in G$ such that
$$[x,_ky]=g_k(x,y)x^{(-1)^k}h_k(x,y)=f_k(x,y)x^{y^k},$$ and $$g_k(x,y)\in\langle x^y,\dots,x^{y^{k-1}}\rangle, \;\; h_k(x,y)\in \langle x^y,\dots,x^{y^{k-1}},x^{y^k}\rangle$$ and $f_k(x,y)\in \langle x,x^y,\dots,x^{y^{k-1}}\rangle$.
\item If $[x,_n y]=1$, then $$\langle x\rangle^{\langle y\rangle}=\langle x,[x,y],\dots,[x,_{n-1}y] \rangle=\langle x,x^y,\dots,x^{y^{n-1}}\rangle.$$
\end{enumerate}
\end{lemma}
\begin{proof}
(1) \; Using $[x,_{k+1} y]=[x,_k y]^{-1}[x,_k y]^{y}$, the proof follows from an easy induction on $k$.\\
(2) \; Let $H=\langle x,[x,y],\dots,[x,_{n-1}y] \rangle$ and $K=\langle x,x^y,\dots,x^{y^{n-1}}\rangle$. Since $x$ belongs to both $H$ and $K$ and they are
 contained in $\langle x\rangle ^{\langle y\rangle}$, it is enough to show that both of $H$ and $K$ are normal subgroups of $\langle x,y\rangle$.
 To prove the latter, it is sufficient to show that
 $[x,_ky]^y,[x,_k y]^{y^{-1}}\in H$ for all $k\in\{0,1,\dots,n-1\}$ and $x^{y^{-1}},x^{y^{n}} \in K$.\\
 We first show the former. Since $[x,_k y]^y=[x,_k y][x,_{k+1}y]$ and $[x,_ny]=1$, $[x,_k y]^y \in H$ for all $k\in\{0,\dots,n-1\}$.
 Now  $[x,_{n-1} y]^{y^{-1}}=[x,_{n-1}y] \in H$ and assume by  a backward induction on $K$ that $[x,_k y]^{y^{-1}}\in H$ for $k<n-1$. Then, as
 $[x,_{k-1} y]^{y^{-1}}=[x,_{k-1} y][x,_k y]^{-y^{-1}}$, we have $[x,_{k-1} y]^{y^{-1}}\in H$. This shows $H$ is also invariant under conjugation of $y^{-1}$.\\
 Now we prove $x^{y^{-1}},x^{y^{n}} \in K$. From part (1) and $[x,_n y]=1$, it follows that $x^{y^n}=f_n(x,y)^{-1}\in K$ and
 $x=\big(g_n(x,y)^{-1}h_n(x,y)^{-1}\big)^{(-1)^n}\in \langle x^y,\dots,x^{y^n}\rangle$. By conjugation of $y^{-1}$, it now follows from the latter that $x^{y^{-1}}\in K$. This completes the proof.
\end{proof}
This lemma is very useful to study Engel groups. We do not know where is its origin and who first noted, however part 2 of Lemma \ref{lem:x^y} has already appeared as Exercise 12.3.6 of the first edition of \cite{Robinson96} published in 1982 and   Rhemtulla and Kim \cite{RhemtullaKim95} groups $G$ having the property that $\langle x\rangle^{\langle y\rangle}$ is finitely generated for all $x,y\in G$  called restrained groups and if there is a bound on the number of generators of such subgroups, they called $G$ strongly restrained. \\
Now we can prove our claim.
\begin{prop}
Let $G$ be any group, $H$ a normal subgroup of $G$ and $K$ a normal subgroup of $H$  such that $M=H/K$ is abelian and suppose further that $H=\langle H\cap R(G)\rangle$. Then $M$ is a restrained $\mathbb{Z}G$-module.
\end{prop}
\begin{proof}
We have to prove that $$S=\frac{\langle x_1^{k_1}\cdots x_n^{k_1}\rangle^{\langle g\rangle} K}{K}$$ is a finitely generated abelian group for any $x_1,\dots,x_n\in R(G)\cap H$, $k_1,\dots,k_n\in \mathbb{Z}$ and any $g\in G$.
Clearly $S$ is a subgroup of  $$L=\frac{\langle x_1,\dots ,x_n\rangle^{\langle g\rangle} K}{K}.$$ Now since $x_i\in R(G)$, part 2 of Lemma \ref{lem:x^y} implies that $\langle x_i\rangle^{\langle g\rangle}$ is finitely generated for each $i$ and so $L$ is a finitely generated abelian group. Thus $S$ is also finitely generated. This completes the proof.
\end{proof}
\begin{thm}[Robinson] {\rm \cite[Theorem 7.34]{Robinson72-2}}
Let $G$ be a radical group. Then $R(G)$ is a locally nilpotent subgroup of $G$. Furthermore, $R(G)=\varrho(G)$ if and only if $R(G)$ is a Gruenberg group.
\end{thm}
A group $G$ is called an $SN^*$-group, if $G$ admits an ascending series whose factors are abelian.
\begin{co}[Robinson] {\rm \cite[Corollary 1, p. 60 of Part II]{Robinson72-2}}
If $G$ is an $SN^*$-group, then $R(G)=\varrho(G)$.
\end{co}
\begin{co}[Robinson] {\rm \cite[Corollary 2, p. 60 of Part II]{Robinson72-2}}
In an arbitrary group the right Engel elements that lie in the final term of the upper Hirsch-Plotkin series from a subgroup.
\end{co}
\begin{qu}[Robinson]{\rm \cite[p. 63 of Part II]{Robinson72-2}}
Let $G$ be an $SN^*$-group. Is it true that $\overline{R}(G)=\overline{\varrho}(G)$?
\end{qu}
\begin{thm}[Gruenberg] {\rm \cite{Gruenberg66}} Let $R$ be a commutative Noetherian ring with identity and $G$ be a group of $R$-automorphisms of a finitely
generated $R$-module. Then $R(G)=\overline{\zeta}(G)$ and $\overline{R}(G)=\zeta_\omega(G)$.
\end{thm}
\begin{thm}[Wehrfritz]{\rm \cite[4.4]{Wehrfritz2000}}
Let $G$ be a group of finitary automorphisms of a module  over
a commutative ring  with identity. Then $R(G)=\varrho(G)$ and $\overline{R}(G)=\overline{\varrho}(G)$.
\end{thm}
Wehrfritz \cite{Wehrfritz92} has also studied the Engel structure of certain linear groups over skew fields.
\subsubsection{Right Engel Elements in Groups Satisfying Certain Min or Max Conditions}
\begin{thm}[Peng]{\rm \cite{Peng66}}\label{Peng66}
Let $G$ be a group which satisfies the maximal condition on its abelian subgroups. Then $R(G)=\overline{R}(G)=\overline{\zeta}(G)$.
\end{thm}
\begin{co}
Let $G$ be a locally finite group. Then $R(G)$ is a subgroup of $HP(G)$.
\end{co}
\begin{proof}
Let $a,b\in R(G)$ and $x\in G$. Then $H=\langle a,b,x\rangle$ is a finite group and so by Theorem \ref{Peng66}, $ab^{-1}\in R(H)$. Hence $[ab^{-1},_k x]=1$ for some integer $k\geq 0$. Now Corollary \ref{co:lfl} and Theorem \ref{t:He} complete the proof.
\end{proof}
\begin{thm}[Plotkin] {\rm \cite{Plotkin58}}
Let $G$ be a group having an ascending series whose factors satisfy Max locally. Then $R(G)$ is a subgroup of $G$.
\end{thm}
\begin{thm}[Held] {\rm \cite{Held66}}
Let $G$ be a group satisfying minimal condition on its abelian subgroups. Then $\overline{R}(G)=\zeta_\omega(G)$.
\end{thm}
\begin{thm}[Martin \& Pamphilon]{\rm \cite[Theorem (iii), (iv)]{MartinPamphilon}}
Let $G$ be a group satisfying minimal condition on those subgroups which can be generated by their left Engel elements. Then $R(G)=\overline{\zeta}(G)$ and $\overline{R}(G)=\zeta_\omega(G)$.
\end{thm}
As far as we know the following result is still unpublished.
\begin{thm}[Martin]{\rm \cite[p. 56 of Part II]{Robinson72-2}}
Let $G$ be a group satisfying minimal condition on its abelian subgroups. Then $R(G)=\overline{\zeta}(G)$.
\end{thm}
\begin{qu}
Let $G$ be an $\mathfrak{M}_c$-group.
\begin{enumerate}
\item Is it true that $R(G)=\varrho(G)$?
\item Is it true that $\overline{R}(G)=\overline{\varrho}(G)$?
\end{enumerate}
\end{qu}
\begin{qu}
 Is $R(G)$ a subgroup for groups $G$ of finite rank? Is $R(G)$  a subgroup for groups $G$ with finite abelian subgroup rank?
\end{qu}
\subsubsection{Right $k$-Engel elements}
For any group $G$, $R_1(G)=Z(G)$.
\begin{thm}[Levi \& W.~P. Kappe] {\rm \cite{Levi42}, \cite{Kappe61}}\label{Levi-Kappe}
Let $G$ be a group, $a\in R_2(G)$ and $x,y,z\in G$.
\begin{enumerate}
\item $a\in L_2(G)$ so that $R_2(G)\subseteq L_2(G)$ and $\langle a\rangle^G$ is an abelian group.
\item $\langle a\rangle^G\subseteq R_2(G)$.
\item $[a,x,y]=[a,y,x]^{-1}$.
\item $[a,[x,y]]=[a,x,y]^2$.
\item $[a^2,x,y,z]=[a,x,y,z]^2=1$ so that $a^2\in \zeta_3(G)$.
\item $[a,[x,y],z]=1$.
\end{enumerate}
\end{thm}
W.~P. Kappe  proved explicitly in \cite{Kappe61} that $R_2(G)$ is a characteristic subgroup for any group $G$.
\begin{thm}[W.~P. Kappe] {\rm \cite{Kappe61}}
Let $G$ be a group. Then $R_2(G)$ is a characteristic subgroup of $G$.
\end{thm}
\begin{proof}
As $R_2(G)$ is invariant under automorphisms of $G$, it is  enough to show that $R_2(G)$ is a subgroup. Let $a,b\in R_2(G)$ and $x\in G$. Then
\begin{align*}
[ab^{-1},_2 x]&=[[a,x]^{b^{-1}}[b,x]^{-b^{-1}},x]\\
&=[[a,x][b,x]^{-1},x[x,b]]^{b^{-1}}\\
&=1
\end{align*}
by parts 3 and 4 of Theorem \ref{Levi-Kappe}. Hence $ab^{-1}\in R_2(G)$.
\end{proof}
\begin{thm}[Newell] {\rm \cite{Newell96}}\label{t:Newell96}
Let $G$ be any group and $x\in R_3(G)$. Then $\langle x\rangle^G$ is a nilpotent group of class at most $3$.
\end{thm}
An essential ingredient to proving Theorem \ref{t:Newell96} was to show $\langle a,b,x\rangle$ is nilpotent for all $a,b\in R_3(G)$ and  $x\in G$. \\
The following asks of a similar property mentioned in Theorem \ref{Levi-Kappe} of right $2$-Engel elements for right $3$-Engel ones.
\begin{qu}
Let $G$ be an arbitrary group and $a\in R_3(G)$. Are there positive integers $n$ and $m$ such that  $a^{m}\in\zeta_n(G)$?
\end{qu}
\begin{thm}[Macdonald]{\rm \cite{Macdonald}}\label{Mac}
There is a finite $2$-group $G$ containing an element $a\in R_3(G)$ such that $a^{-1}\not\in R_3(G)$ and $a^2\not\in R_3(G)$.
\end{thm}
On the positive side, we have the following results.
\begin{thm}[Heineken]{\rm \cite{Heineken61}}
If $A$ is the subset of a group $G$ consisting of all elements $a$
such that both $a$ and $a^{-1}$ belongs to $R_3(G)$, then $A$ is a subgroup if either
$G$ has no element of order $2$ or $A$ consists only of elements
having finite odd order.
\end{thm}
\begin{thm}[Abdollahi \& Khosravi]{\rm \cite{AbdollahiKhosravi2}}
Let $G$ be a group such that $\gamma_5(G)$ has no element of order $2$. Then $R_3(G)$ is a subgroup of $G$.
\end{thm}
\begin{proof}
It follows from detail information of the subgroup $\langle a,b,x\rangle$ where $a,b\in R_3(G)$ and $x\in G$.
\end{proof}
L.-C. Kappe and Ratchford \cite{KappeRatchford99} have shown that $R_n(G)$ is a subgroup of $G$ whenever $G$ is metabelian or center-by-metabelian with certain
 extra properties.

Nickel \cite{Nickel97}  generalized Macdonald's example (Theorem \ref{Mac}) to all right $n$-Engel elements for any
$n\geq 3$ by proving that  there is a nilpotent group of class $n+2$ containing a right $n$-Engel
element $a$ and an element $b$ such that  $[a^{-1},_n b]=[a^2,_n b]$ is non-trivial.

Using the  group constructed by Newman and Nickel \cite{NewmanNickel94}, it is shown in \cite{AbdollahiKhosravi2} that
there is a group containing a right $n$-Engel element $x$ such that $x^k$ and $x^{-1}$ are not in $R_n(G)$ for all $k\geq 2$.

By Theorem \ref{t:Newell96} of Newell, we know that   $R_3(G)\subseteq Fitt(G)$ for any group $G$ and on the other hand  Gupta and Levin \cite{GuptaLevin} have  shown that the normal closure of an element in a $5$-Engel group need not be
nilpotent (see also \cite[p. 342]{Vaughan-Lee2007}).
\begin{thm}[Gupta \& Levin] {\rm  \cite{GuptaLevin} }\label{thm:GuptaLevin}
For each prime $p\geq 3$,  let $G$ be the free nilpotent of class $2$ group of exponent $p$ and of countably infinite rank. Let $M_p$ be the  the multiplicative group  of $2\times 2$ matrices over the group ring $\mathbb{Z}_pG$  of the form
 $ \begin{pmatrix}
                                        g & 0 \\
                                        r & 1 \\
                                      \end{pmatrix}$, where $g\in G$ and $r\in\mathbb{Z}_pG$. Then
 the  group $M_p$  has the following properties:
 \begin{enumerate}
\item $M_p$ has exponent $p^2$ and $\gamma_3(M_p)$ has exponent $p$;
\item $M_p$ is abelian-by-$($nilpotent of class $2)$;
\item $M_p$ is a $(p+2)$-Engel group;
\item $M_p$ has an element  whose normal closure in $M_p$ is not nilpotent.
\end{enumerate}
\end{thm}
This result of Gupta and Newman  raises naturally the following question.
\begin{qu}
Let $n$ be a positive integer. For which primes $p$, there exists a soluble $n$-Engel $p$-group $M(n,p)$ which is not a Fitting group?
\end{qu}
By Gupta-Levin's Theorem \ref{thm:GuptaLevin}, for all $p\geq 3$ and for all $n\geq p+2$,  $M(n,p)$ exists. We observed that  $M(6,2)$ also exists. In fact, a similar construction of Gupta and Levin gives $M(6,2)$.
\begin{prop}
Let $G$ be the free nilpotent of class $2$ group of exponent $4$ and of countably infinite rank.
Let $M$ be the  the multiplicative group  of $2\times 2$ matrices over the group ring $\mathbb{Z}_2G$  of the form
 $ \begin{pmatrix}
                                        g & 0 \\
                                        r & 1 \\
                                      \end{pmatrix}$, where $g\in G$ and $r\in\mathbb{Z}_2G$. Then
 the  group $M$  has the following properties:
 \begin{enumerate}
\item $M$ has exponent $8$ and $\gamma_3(M)$ has exponent $2$;
\item $M$ is abelian-by-$($nilpotent of class $2)$;
\item $M$ is a $6$-Engel group;
\item $M$ has an element  whose normal closure in $M$ is not nilpotent.
\end{enumerate}
\end{prop}
\begin{proof}
We first observed that the elements $ \begin{pmatrix}
                                        1 & 0 \\
                                        r & 1 \\
                                      \end{pmatrix}$
 of $M$ constitute an elementary abelian $2$-subgroup $K$. Since $exp(G)=4$ and $cl(G)=2$, it follows that $\gamma_3(M)\leq K$ and $M^4\leq K$. Thus we have  proved parts 1 and 2. For the proof of part 3, we first note that if $A=\begin{pmatrix}
                                                                          1 & 0 \\
                                                                          s & 1 \\
                                                                        \end{pmatrix}$ and $B=\begin{pmatrix}
                                                                                                g & 0 \\
                                                                                                r & 1 \\
                                                                                              \end{pmatrix}$ then $A^{-1}=\begin{pmatrix}
                                                                          1 & 0 \\
                                                                         -s & 1 \\
                                                                        \end{pmatrix}$ and $B^{-1}=\begin{pmatrix}
                                                                                                g^{-1} & 0 \\
                                                                                                -rg^{-1} & 1 \\
                                                                                              \end{pmatrix}$. Thus the commutator $[A,B]=\begin{pmatrix}
                                                                                                1 & 0 \\
                                                                                                s(g-1) & 1 \\
                                                                                              \end{pmatrix}$, and by interation
 $$[A,_4 B]= \begin{pmatrix}
                                                                                                1 & 0 \\
                                                                                                s(g-1)^4 & 1 \\
                                                                                              \end{pmatrix}=\begin{pmatrix}
                                                                                                1 & 0 \\
                                                                                                0 & 1 \\
                                                                                              \end{pmatrix}.$$
Since every element of $\gamma_3(M)$ is of the form $A$, it follows that $M$ is a $6$-Engel group.\\
Finally, for the proof of part 4, we first note that for $X_i=\begin{pmatrix}
                                                                                                x_i & 0 \\
                                                                                                1 & 1 \\
                                                                                              \end{pmatrix}$ $i\geq 0$, the commutator $[X_i,X_j]$ is of the form $\begin{pmatrix}
                                                                                                [x_i,x_j] & 0 \\
                                                                                                r & 1 \\
                                                                                              \end{pmatrix}$. Thus if $Y=\begin{pmatrix}
                                                                                                1 & 0 \\
                                                                                                1 & 1 \\
                                                                                              \end{pmatrix}$ then
$$ [Y,[X_0,X_1],\dots,[X_0,X_m]]=\begin{pmatrix}
                                                                                                1 & 0 \\
                                                                                                u_m & 1 \\
                                                                                              \end{pmatrix},$$
where $u_m=([x_0,x_1]-1)\cdots ([x_0,x_m]-1)$ and it is a non-zero element of $\mathbb{Z}_2G$ for all $m\geq 1$. It now follows that the normed closure of $X_0$ in $M$ is not nilpotent. This completes the proof.

\end{proof}
\begin{qu}
Does there exist a soluble $5$-Engel $2$-group  which is not a Fitting group?
\end{qu}
It follows that
$R_n(G)\nsubseteq Fitt(G)$ for $n\geq 5$. The following question naturally
arises.
\begin{qu}\label{qu:RnFi} What are the least positive
integers $n$, $m$ and $\ell$ such that
\begin{enumerate} \item $R_n(G_1)\nsubseteq Fitt(G_1)$ for some group $G_1$? \item $R_m(G_2)\nsubseteq B(G_2)$ for some group $G_2$? \item $R_\ell(G_3)\nsubseteq HP(G_3)$ for some group $G_3$?
\end{enumerate}
\end{qu}
Therefore, to find integer $n$ in Question \ref{qu:RnFi} we have to answer the
following.
\begin{qu} Let $G$ be an arbitrary group. Is it true that
$R_4(G)\subseteq Fitt(G)$?
\end{qu}
For right $4$-Engel elements there are  some results.
\begin{thm}[Abdollahi \& Khosravi]\rm{\cite[Theorem 1.3]{AbdollahiKhosravi1}}
Let $G$ be any group. If $a\in G$  and both $b,b^{-1}\in R_4(G)$, then
$\langle a,a^b\rangle$ is nilpotent of class at most $4$.
\end{thm}
\begin{thm}[Abdollahi \& Khosravi]{\rm \cite{AbdollahiKhosravi2}} \label{thm:AbdollahiKhosraviR4}
Let $G$ be a  $\{2,3,5\}'$-group such that $\langle a,b,x\rangle$ is nilpotent for all $a,b\in R_4(G)$ and any $x\in G$. Then $R_4(G)$
is a subgroup of $G$.
\end{thm}
An important tool in the proof of Theorem \ref{thm:AbdollahiKhosraviR4} is the nilpotent quotient algorithm as implemented in the NQ package \cite{Nickel98} of {\sf GAP} \cite{GAP}. Indeed we need to know the structure of the largest nilpotent quotient of a nilpotent $\{2,3,5\}'$-group generated by  two right $4$-Engel elements and an arbitrary element. It is a byproduct of the proof of Theorem \ref{thm:AbdollahiKhosraviR4} that
\begin{co}
Let $G$ be a $\{3,5\}'$-group such that $\langle a,x\rangle$ is nilpotent for all $a\in R_4(G)$ and any $x\in G$. Then $R_4(G)$ is  inverse closed.
\end{co}
\begin{co}[Abdollahi \& Khosravi]{\rm  \cite{AbdollahiKhosravi2}}\label{co:AbdollahiKhosraviR4nil}
Let $G$ be a  $\{2,3,5\}'$-group such that $\langle a,b,x\rangle$ is nilpotent for all $a,b\in R_4(G)$ and for any $x\in G$. Then $R_4(G)$ is a nilpotent group of class at most $7$. In particular,  the normal closure of every right $4$-Engel element of  $G$ is nilpotent
of class at most $7$.
\end{co}
\begin{proof}
By  Theorem \ref{thm:AbdollahiKhosraviR4}, $R_4(G)$ is a subgroup of $G$ and so it
is a $4$-Engel group. Now it follows from a result of Havas and Vaughan-Lee \cite{HavasVaughan-Lee2005} that $4$-Engel groups are locally nilpotent, $R_4(G)$ is locally nilpotent. By \cite{Traustason95}, we know that every
locally nilpotent 4-Engel $\{2,3,5\}'$-group is nilpotent of class at most $7$.
Therefore $R_4(G)$ is nilpotent of class at most $7$. Since $R_4(G)$ is a normal set, the second part follows easily.
\end{proof}
\begin{qu}
Let $$\mathfrak{C}_4=\big\{ cl\big(\langle x\rangle^G\big) \;|\; G\;\text{is a group such that}\; x\in R_4(G) \; \text{and} \; \langle x\rangle ^G \;\text{is nilpotent}\big\},$$ where $cl(X)$ denotes the nilpotent class of a nilpotent group $X$.
\begin{enumerate}
\item Is  the set $\mathfrak{C}_4$ bounded?
\item If  the part 1 has  positive answer, what is the maximum of $\mathfrak{C}_4$? Is it $4$?
\end{enumerate}
\end{qu}
\begin{thm}[Abdollahi \& Khosravi]{\rm \cite{AbdollahiKhosravi2}}
In any $\{2,3,5\}'$-group, the normal closure of any right $4$-Engel element is nilpotent if and only if every $3$-generator subgroup in which two of the generators can be chosen to be  right $4$-Engel, is nilpotent.
\end{thm}
\begin{proof}
By Corollary \ref{co:AbdollahiKhosraviR4nil}, it is enough to show that  a $\{2,3,5\}'$-group $H=\langle a,b,x\rangle$ is nilpotent whenever $a,b\in R_4(H)$, $x\in H$ and both $\langle a\rangle^H$ and $\langle b\rangle ^H$ are nilpotent. Consider the subgroup $K=\langle a\rangle ^H\langle b\rangle^H$ which is nilpotent by Fitting's theorem.  We have $K=\langle a,b\rangle^{\langle x\rangle}$ and since $a$ and $b$ are both right Engel, we have (see e.g., \cite[Exercise 12.3.6, p. 376]{Robinson96} that both
$\langle a\rangle^{\langle x\rangle}$ and $\langle b\rangle^{\langle x\rangle}$ are finitely generated. Thus $K$ is also finitely generated.
Hence $H$  satisfies maximal condition on its subgroups. Now Theorem \ref{Peng66}  completes the proof.
\end{proof}
It may be interesting to know that  in a nilpotent $\{2,3,5\}$-free group every two right $4$-Engel element with an arbitrary element always generate a $4$-Engel group.
\begin{thm}
Let $G$ be any group, $a,b\in R_4(G)$ and $x\in G$. If $\langle a,b,x\rangle$ is a nilpotent $\{2,3,5\}$-free group, then it is a $4$-Engel group of class at most $7$.
\end{thm}
\begin{proof}
A proof is similar to one of \cite[Theorem 3.4]{AbdollahiKhosravi2}.
\end{proof}
The following problem is the right analog of Problem \ref{prob:ll}.
\begin{prob} \label{prob:RR}Let $G$ be an arbitrary group. Find all  pairs $(n,m)$ of positive integers such that,
   $xy\in \overline{R}(G)$ whenever $x\in R_n(G)$ and $y\in R_m(G)$.
\end{prob}
Since the set of right $2$-Engel elements is a subgroup, $(2,2)$ belongs to the solutions of Problem \ref{prob:RR}. We show that $(2,3)$ and $(3,3)$ also belong to the solutions. In fact we prove more.
\begin{prop}
Let $G$ be an arbitrary  group, $a\in R_2(G)$ and $b,c\in R_3(G)$. Then $ab\in R_3(G)$ and $bc\in R_4(G)$.
\end{prop}
\begin{proof}
By \cite{Newell96} $K=\langle b,c,x\rangle$ is nilpotent for all $x\in G$. In particular, $H=\langle a,b,x\rangle$ is also nilpotent for all $x\in G$. Now, thanks to the NQ package \cite{Nickel98} of {\sf GAP} \cite{GAP}, one can easily construct the freest  nilpotent groups with the same defining relations as $K$ and $H$. Then   the conclusion can be easily checked through two line commands in {\sf GAP} \cite{GAP}.
\end{proof}
 By using the positive solution of restricted Burnside's problem due to  Zel'manov \cite{Zelmanov91-odd,Zelmanov91-even}, Shalev has proved that:
\begin{thm}[Shalev]{\rm \cite[Proposition D]{Shalev93}}
There is a function $c:\mathbb{N}\times \mathbb{N}\rightarrow \mathbb{N}$ such that for any $d$-generated nilpotent group  $G$  and any normal subgroup $H$ of $G$ with $H\subseteq R_n(G)$, we have $H\subseteq \zeta_{c(n,r)}(G)$.
\end{thm}
We finish by the following question.
\begin{qu}
Are there  functions $c,e:\mathbb{N}\rightarrow \mathbb{N}$ such that for any nilpotent group  $G$  and any normal subgroup $H$ of $G$ with $H\subseteq R_n(G)$, we have $H^{e(n)}\subseteq \zeta_{c(n)}(G)$?
\end{qu}
\noindent\textbf{Acknowledgments.} This survey was completed  during the author's visit to University of Bath in 2009.
The author is very grateful to Department of Mathematical Sciences of University of Bath and specially he wishes to
 thank Gunnar Traustason for their kind hospitality. The author gratefully  acknowledges financial support of University of
 Isfahan  for his sabbatical leave study. This work is also financially  supported by the Center of Excellence for Mathematics, University of Isfahan.

\end{document}